\documentclass{article}

\usepackage[a4paper]{geometry}
\usepackage[english,strings]{babel}
\usepackage[utf8]{inputenc}
\usepackage[T1]{fontenc}
\usepackage{amsmath}
\usepackage{amssymb}
\usepackage{amsthm}
\usepackage{amscd}
\usepackage{tikz}
\usetikzlibrary{cd}
\usepackage[all,cmtip]{xy}
\usepackage{hyperref}
\usepackage{enumitem}
\usepackage{mathrsfs}
\usepackage{mathtools}
\usepackage{authblk}

\theoremstyle{definition}
\newtheorem{theorem}			     {Theorem}		[section]
\newtheorem{proposition}  [theorem]	 {Proposition}	
\newtheorem{lemma}	      [theorem]  {Lemma}		

\newtheorem{remark} 	  [theorem]  {Remark}

\newtheorem{corollary}	  [theorem]  {Corollary}

\title{Double groupoids and $2$-groupoids in regular Mal'tsev categories}
\author[a]{Nadja Egner}
\author[a]{Marino Gran}
\affil[a]{\small\textit{Institut de Recherche en Math\'ematique et Physique, Universit\'e catholique de Louvain, Chemin du Cyclotron 2, 1348 Louvain-la-Neuve, Belgium}}
\affil[ ]{nadja.egner@uclouvain.be, marino.gran@uclouvain.be}
\date{}

\begin{document}

\maketitle

\begin{center}
	\textit{Dedicated to Robert Paré at the occasion of his 80th birthday}
\end{center}

\vspace{0.5cm}

\begin{abstract}
	We prove that the category 2-$ \mathrm{Grpd}(\mathscr{C}) $ of internal 2-groupoids is a Birkhoff subcategory of the category $ \mathrm{Grpd}^2(\mathscr{C}) $  of double groupoids in a regular Mal'tsev category $\mathscr{C}$ with finite colimits, and we provide a simple description of the reflector. In particular, when $\mathscr{C}$ is a Mal'tsev variety of universal algebras, the category 2-$ \mathrm{Grpd}(\mathscr{C}) $  is also a Mal'tsev variety, of which we describe the corresponding algebraic theory. When $\mathscr{C}$ is a naturally Mal'tsev category, the reflector from $ \mathrm{Grpd}^2(\mathscr{C}) $ to 2-$ \mathrm{Grpd}(\mathscr{C}) $ has an additional property related to the commutator of equivalence relations. We prove that the category 2-$ \mathrm{Grpd}(\mathscr{C}) $ is semi-abelian when $\mathscr{C}$ is semi-abelian, and then provide sufficient conditions for 2-$ \mathrm{Grpd}(\mathscr{C}) $ to be action representable.
	
	
\end{abstract}

\section*{Introduction}

During the last four decades, internal groupoids in the categories of groups, Lie algebras, commutative algebras and more  general varieties of universal algebras have been shown to form very rich algebraic categories. They have strong connections with commutator theory (see \cite{CPP, Ja-Pe, Pedicchio, Gran-Rosicky, BGJ}, for instance) and are also of interest in homotopical algebra (see \cite{Loday, Porter, Conduche, Everaert-K-T}, for instance).
An internal groupoid in the category $\mathsf{Grp}$ of groups can be equivalently presented as a diagram
\begin{equation}\label{reflexive-graph}
	\begin{tikzcd}
		C_1 \arrow[r, shift left=2, "d"] \arrow[r, shift right=2, "c"'] &C_0 \arrow[l, "e" description]
	\end{tikzcd}
\end{equation}
in $\mathsf{Grp}$ such that $de = 1_{C_0}=ce$, so that it is a reflexive graph, and the kernels $\mathsf{Ker} (d)$ and $\mathsf{Ker}(c)$ of the ``domain'' and ``codomain'' group homomorphisms $d$ and $c$ have trivial commutator: $[\mathsf{Ker}(d), \mathsf{Ker}(c)]= \{1\}$ \cite{Loday}. The  ``composition'' of the internal groupoid is then a group homomorphism $ m: C_1\times_{C_0} C_1\to C_1 $, where $ C_1\times_{C_0} C_1 $ represents the group of pairs of ``composable arrows'', and is defined, for any pair $(f,g) \in C_1\times_{C_0} C_1$, by $m(f,g) = f \cdot {1_Y}^{-1} \cdot g$, where $\cdot$ represents the  multiplication in the ``group of arrows'' $C_1$ and $ 1_Y=e(Y) $, where $ Y=c(f)=d(g) $. A similar description of the category of internal groupoids is given in Mal'tsev varieties \cite{Ja-Pe}, which are the varieties of universal algebras whose theory contains a ternary term $p(x,y,z)$ satisfying the identities $p(x,y,y)=x$ and $p(x,x,y)=y$ \cite{Smith}. In particular any variety whose theory contains the operations and the identities of the theory of groups has such a Mal'tsev term, since it contains the ternary term $p(x,y,z)= x \cdot y^{-1} \cdot z$. In addition quasigroups, loops, Heyting algebras and $\mathsf{MV}$-algebras are Mal'tsev varieties. 

Many exactness properties of the category  $ \mathrm{Grpd}(\mathscr{C}) $ of internal groupoids in a Mal'tsev variety $\mathscr{C}$, and in more general algebraic categories, have 
been established in various articles (see \cite{gran:1999, Pedicchio, Gran-Rosicky, GRT, gran.gray:2021}, for instance).
The notion of regular Mal'tsev category \cite{CLP} generalizes the one of Mal'tsev variety, in the sense that the \emph{syntactic} property of the existence of the ternary Mal'tsev term $p(x,y,z)$ is replaced by the \emph{semantic} property that any internal reflexive relation in it is necessarily an equivalence relation. Besides all the varietal examples recalled above, among regular Mal'tsev categories we find the categories of topological groups, Banach spaces,
$\mathbb{C}^*$-algebras, 
torsion-free (abelian) groups,  cocommutative Hopf algebras over a field,  the dual of the category of sets, any abelian category, and many more (see \cite{BGJ} and the references therein).

In the present work we explore the properties of the categories $ \mathrm{Grpd}^2(\mathscr{C}) $ of internal double groupoids and 2-$ \mathrm{Grpd}(\mathscr{C}) $ of $2$-groupoids in any finitely cocomplete regular Mal'tsev category $\mathscr{C}$. Recall that a double groupoid is simply an internal groupoid in the category $\mathrm{Grpd}(\mathscr{C})$. It has an underlying double reflexive graph that we depict as  
\begin{equation}\label{Double-reflexive}
	\begin{tikzcd}
		C^1_1 \arrow[d, shift right=3, "d_1"'] \arrow[d, shift left=3, "c_1"] \arrow[r, shift left=3, "d^1"] \arrow[r, shift right=3, "c^1"']  &C^1_0 \arrow[d, shift right=3, "d_0"'] \arrow[d, shift left=3, "c_0"] \arrow[l, "e^1" description] \\
		C^0_1 \arrow[u, "e_1" description] \arrow[r, shift left=3, "d^0"] \arrow[r, shift right=3, "c^0"'] &{C^0_0. }\arrow[u, "e_0" description] \arrow[l, "e^0" description]
	\end{tikzcd}
\end{equation}
Such an internal groupoid is a $2$-groupoid precisely when the morphism
$ e^0 $ is an isomorphism, this property expressing the fact that the ``vertical structure'' of the double groupoid is discrete.

We prove that the category 2$\textrm{-Grpd}(\mathscr{C})$
of $2$-groupoids  in $\mathscr{C}$
is a reflective  subcategory of the category
$\mathrm{Grpd}^2(\mathscr{C}) $ of double groupoids in $\mathscr{C}$, whenever $\mathscr{C}$ is a regular Mal'tsev category with finite colimits. The left adjoint $ \mathsf{F}:\mathrm{Grpd}^2(\mathscr{C})\to 2\textrm{-Grpd}(\mathscr{C}) $ of the forgetful functor  $ \mathsf{U}: 2\textrm{-Grpd} \to \mathrm{Grpd}^2(\mathscr{C}) $ is quite simple to describe, thanks to the property that the category $ \mathrm{Grpd}(\mathscr{C}) $ is closed in the category $ \mathrm{RG}(\mathscr{C}) $ of reflexive graphs in $\mathscr{C}$ under regular quotients. It is actually a Birkhoff subcategory (Theorem \ref{Birkhoff}), this latter property meaning that 2$\textrm{-Grpd}(\mathscr{C})$ is also closed in $\mathrm{Grpd}^2(\mathscr{C}) $ under subobjects and regular quotients. In particular, when $\mathscr{C}$ is a Mal'tsev variety, this result implies that 2$\textrm{-Grpd}(\mathscr{C})$ is a subvariety of the Mal'tsev variety $\mathrm{Grpd}^2(\mathscr{C}) $, for which Corollary \ref{theory-variety} provides a precise description of the algebraic theory. In the next section we investigate the same adjunction in the context of naturally Mal'tsev categories \cite{Johnstone}, where the unit components of the previously considered adjunction turn out to have an additional property (Proposition \ref{units-NM}). In Section \ref{Action-R} we restrict ourselves to the case where the base category $\mathscr{C}$ is semi-abelian \cite{JMT}, and we investigate the problem of establishing a sufficient condition for the category 2$\textrm{-Grpd}(\mathscr{C})$ to be action representable in the sense of \cite{BJK}. By using the recent result in \cite{gran.gray:2021} concerning the action representability of the category $\mathrm{Grpd}^2(\mathscr{C}) $, we deduce that 2-$\mathrm{Grpd}(\mathscr{C}) $ is semi-abelian, action representable, algebraically coherent \cite{cigoli.gray.vanderlinden:2015} with normalizers \cite{Gray} 
whenever the base category $\mathscr{C}$ has all these properties (Theorem \ref{main-AR}). This essentially depends on the fact that 2-$\mathrm{Grpd}(\mathscr{C})$ is not only a Birkhoff subcategory of $\mathrm{Grpd}^2(\mathscr{C}) $, but it is also coreflective (Lemma \ref{prop:2-Grpd(C)coreflective}).

\textsc{Acknowledgements.} The authors thank Pierre-Alain Jacqmin and George Janelidze for some very useful discussions on the subject of this article. The authors also thank the anonymous referees, in particular for the comment that led to Remark~\ref{rem:alternativeproof}. The first author's research is funded by a FNRS doctoral grant of the Communaut\'{e} fran\c{c}aise de Belgique. The second author's research was supported by the Fonds de la Recherche Scientifique - FNRS under Grant CDR No. J.0080.23.

\section{Preliminaries}

In this section, we will recall the notions of internal groupoid, Mal'tsev category, semi-abelian category and Birkhoff subcategory, and some of their properties. 

\subsection*{Internal groupoids}

Throughout the remainder of the paper, $ \mathscr{C} $ will always be a finitely complete category. An internal reflexive graph \eqref{reflexive-graph} in $ \mathscr{C} $ 
has an internal category structure if there exists a ``composition'' morphism $ m: C_1\times_{C_0} C_1\to C_1 $, where $ C_1\times_{C_0} C_1 $ is the object part of the pullback of the ``codomain'' morphism $ c $  along the ``domain'' morphism $ d $ as displayed in the diagram 
\begin{equation*}
	\begin{tikzcd}
		C_1\times_{C_0} C_1 \arrow[r, "p_2"] \arrow[d, "p_1"'] &C_1 \arrow[d, "d"]\\
		C_1 \arrow[r, "c"'] &C_0,
	\end{tikzcd}
\end{equation*}
that satisfies the usual axioms(see \cite{MacLane}). An internal category, i.e. a reflexive graph with a fixed internal category structure, is called an internal groupoid if, in addition, there exists an ``inverse'' morphism $ i: C_1\to C_1 $ satisfying the usual axioms. One can show that, if  $ i $ exists, it is necessarily unique. 

A morphism from a reflexive graph $ \mathbb{C} $ to another reflexive graph $ \mathbb{D} $ is given by a pair $ (f_0,f_1) $ of morphisms in $ \mathscr{C} $ as displayed in the diagram 
\begin{equation}\label{reflexive-graph-morphism}
	\begin{tikzcd}
		C_1 \arrow[d, "f_1"'] \arrow[r, shift left=2, "d"] \arrow[r, shift right=2, "c"'] &C_0 \arrow[d, "f_0"] \arrow[l, "e" description]\\
		D_1 \arrow[r, shift left=2, "d"] \arrow[r, shift right=2, "c"'] &D_0, \arrow[l, "e" description]
	\end{tikzcd}
\end{equation}
such that $ f_0d=df_1 $, $ f_0c=cf_1 $ and $ f_1e=ef_0$. An internal functor from an internal category $ \mathbb{C} $ to another internal category $ \mathbb{D} $ is a morphism of the underlying reflexive graphs of $ \mathbb{C} $ and $ \mathbb{D} $ that also preserves the composition maps in the usual sense. 

We denote by $ \mathrm{RG}(\mathscr{C}) $ the category whose objects are the reflexive graphs in $ \mathscr{C} $ and whose morphisms are the morphisms of reflexive graphs just defined. We denote by $ \mathrm{Cat}(\mathscr{C}) $ the category whose objects are the internal categories in $ \mathscr{C} $ and whose morphisms are the internal functors. The full subcategory of $ \mathrm{Cat}(\mathscr{C}) $ whose objects are the internal groupoids in $ \mathscr{C} $ is denoted by $ \mathrm{Grpd}(\mathscr{C}) $. It is well-known that an internal functor between two internal groupoids automatically preserves the inverse morphisms. 

When $\mathscr{C}$ is finitely complete, so is the category 
$ \mathrm{Grpd}(\mathscr{C}) $, and its limits are computed ``levelwise''.
One then defines the category  $ \mathrm{Grpd}^2(\mathscr{C}):= \mathrm{Grpd}(\mathrm{Grpd}(\mathscr{C})) $ of internal double groupoids in $ \mathscr{C} $.  Every double groupoid $ \mathbb{C} $ in $ \mathscr{C} $ has an underlying double reflexive graph \eqref{Double-reflexive}. 
We denote by $ 2\textrm{-Grpd}(\mathscr{C}) $ the full subcategory of $ \mathrm{Grpd}^2(\mathscr{C}) $ whose objects are the internal $2$-groupoids, these latter being  characterized by the property that $e^0$ is an isomorphism. When this is the case, one clearly has that $ d^0=c^0=(e^0)^{-1} $.

\subsection*{Mal'tsev categories and semi-abelian categories}

As already mentioned in the Introduction, a Mal'tsev category is a finitely complete category such that any internal reflexive relation in it is an equivalence relation. One of their interesting properties is that any reflexive graph in a Mal'tsev category admits at most one internal category structure, and that any internal category is automatically an internal groupoid \cite{CPP}. This means that the categories $ \mathrm{Cat}(\mathscr{C}) $ and $ \mathrm{Grpd}(\mathscr{C}) $ are isomorphic. Moreover, the category $\mathrm{Grpd}(\mathscr{C}) $ of internal groupoids in a Mal'tsev category $ \mathscr{C} $ is a full subcategory of $ \mathrm{RG}(\mathscr{C}) $ since any morphism of reflexive graphs between two internal groupoids automatically  preserves the composition \cite{CPP}.

Let us recall that a regular category $ \mathscr{C} $ is a finitely complete category that has coequalizers of kernel pairs and whose regular epimorphisms are pullback stable. This implies that any morphism in $\mathscr{C}$ can be factorized as a regular epimorphism followed by a monomorphism, and that these factorizations are pullback stable. If $ \mathscr{C} $ is a regular Mal'tsev category, then $ \mathrm{Grpd}(\mathscr{C}) $ is a regular Mal'tsev category \cite{gran:1999}. 
Indeed, this mainly follows from the observation that a morphism $ (f_0,f_1)$ as in diagram \eqref{reflexive-graph-morphism}, where $\mathbb C$ and $\mathbb D$ now represent the reflexive graphs underlying two internal groupoids, is a regular epimorphism in $ \mathrm{Grpd}(\mathscr{C}) $ if and only if it is so in $ \mathrm{RG}(\mathscr{C}) $, i.e., both $f_0$ and $f_1 $ are regular epimorphisms in $ \mathscr{C} $. As a consequence, the category $ \mathrm{Grpd}(\mathscr{C}) $ is closed in $ \mathrm{RG}(\mathscr{C}) $ under regular quotients, meaning that, if $ (f_0,f_1) $ is a regular epimorphism of reflexive graphs whose domain is an internal groupoid, then its codomain is also an internal groupoid. Recall that an exact category in the sense of Barr \cite{Barr} is a regular category in which every  equivalence relation is effective, i.e., it is a kernel pair. If $\mathscr{C}$ is an exact Mal'tsev category, then so is the category $ \mathrm{Grpd}(\mathscr{C}) $. 

A semi-abelian category is a category that is pointed, exact, protomodular and has binary coproducts. In this context protomodularity is  equivalent to the validity of the Split Short Five Lemma, i.e., given a diagram
\begin{equation*}
	\begin{tikzcd}
		A \arrow[d, equal] \arrow[r, "k"] & B \arrow[d, "b"'] \arrow[r, shift left=1, "p"] & C \arrow[d, equal] \arrow[l, shift left=1, "s"]\\
		A \arrow[r, "k'"'] & B' \arrow[r, shift left=1, "p'"] & C, \arrow[l, shift left=1, "s'"]
	\end{tikzcd}
\end{equation*}
where $ ps=1_C $, $ p's'=1_{C} $, $ k=\mathrm{ker}(p) $, $ k'=\mathrm{ker}(p') $, $ bk=k' $, $ p'b=p $ and $ bs=s' $, then $ b $ is an isomorphism.  Any semi-abelian category is a finitely cocomplete and Mal'tsev category. Examples of semi-abelian categories are given by the categories of groups, rings, Lie algebras over a commutative ring, cocommutative Hopf algebras over a field, or $ \mathbb{C}^*$-algebras. If $ \mathscr{C} $ is semi-abelian, then $ \mathrm{Grpd}(\mathscr{C}) $ is also semi-abelian. 

\subsection*{Birkhoff subcategories}

In the following, let $ \mathscr{X} $ be a replete, i.e. closed under isomorphisms, full reflective subcategory of a category $ \mathscr{C} $. It is called a Birkhoff subcategory of $ \mathscr{C} $ if it is closed both under subobjects and regular quotients. If $ \mathscr{C} $ is a regular category, then $ \mathscr{X} $ is closed under subobjects in $ \mathscr{C} $ if and only if it is regular epi-reflective, i.e., each component of the unit of the reflection is a regular epimorphism in $ \mathscr{C} $. If $ \mathscr{C} $ is a regular category and $ \mathscr{X} $ is closed under subobjects in $ \mathscr{C} $, then $ \mathscr{X} $ is regular. In this case a morphism in $ \mathscr{X} $ is a regular epimorphism (monomorphism) if and only if it is a regular epimorphism (monomorphism) in $ \mathscr{C} $. If $ \mathscr{C} $ is exact and $ \mathscr{X} $ is a Birkhoff subcategory of $ \mathscr{C} $, then $ \mathscr{X} $ is exact. 

The terminology ``Birkhoff subcategory'' is justified by the classical theorem due to G. Birkhoff, asserting that a class of universal algebras $\mathscr{A}$ is a subvariety of a variety $\mathscr{B}$ of universal algebras if and only if it is closed in $\mathscr{B}$ under products, subalgebras and homomorphic images. Accordingly,  the Birkhoff subcategories of a given variety $\mathscr{B}$ are exactly its subvarieties. 

We note that a subcategory $ \mathscr{X} $ as above is a regular Mal'tsev category whenever $ \mathscr{C} $ is and $ \mathscr{X} $ is closed under subobjects in $ \mathscr{C} $. It is exact Mal'tsev (resp., semi-abelian) whenever $ \mathscr{C} $ is and $ \mathscr{X} $ is a Birkhoff subcategory of $ \mathscr{C} $. Moreover, if $ \mathscr{C} $ is a regular Mal'tsev category with coequalizers, then $ \mathrm{Grpd}(\mathscr{C}) $ is a Birkhoff subcategory of $ \mathrm{RG}(\mathscr{C}) $.

\section{The Birkhoff subcategory of internal 2-groupoids}

\begin{theorem}\label{Birkhoff}
	Let $ \mathscr{C} $ be a finitely cocomplete regular Mal'tsev category. Then $ 2\text{-}\mathrm{Grpd}(\mathscr{C}) $ is a Birkhoff subcategory of $ \mathrm{Grpd}^2(\mathscr{C}) $.  
\end{theorem}
\begin{proof}
	Let us define the reflector $ \mathsf{F}:\mathrm{Grpd}^2(\mathscr{C})\to 2\textrm{-Grpd}(\mathscr{C}) $ as follows. For a double groupoid $ \mathbb{C} $ in $ \mathscr{C} $, we define $ \mathsf{F}(\mathbb{C}) $ as the $2$-groupoid $ \mathbb{F} $ in the front part of the following diagram: 
	\begin{equation}\label{reflection}
		\begin{tikzcd}
			{C^1_1} && {C^1_0} \\
			\\
			{C^0_1} && {C^0_0} \\
			&& {F^1_1} && {F^1_0} \\
			\\
			&& {F^0_0} && {F^0_0} 
			\arrow["{d^1}", shift left=2, from=1-1, to=1-3]
			\arrow["{c^1}"', shift right=2, from=1-1, to=1-3]
			\arrow["{e^1}"{description}, from=1-3, to=1-1]
			\arrow["{c_1}", shift left=2, from=1-1, to=3-1]
			\arrow["{d_1}"', shift right=2, from=1-1, to=3-1]
			\arrow["{e_1}"{description}, from=3-1, to=1-1]
			\arrow["{d^0}", shift left=2, from=3-1, to=3-3]
			\arrow["{c^0}"', shift right=2, from=3-1, to=3-3]
			\arrow["{e^0}"{description}, from=3-3, to=3-1]
			\arrow["{c_0}", shift left=2, from=1-3, to=3-3]
			\arrow["{d_0}"', shift right=2, from=1-3, to=3-3]
			\arrow["{e_0}"{description}, from=3-3, to=1-3]
			\arrow["{\delta^1}", shift left=2, from=4-3, to=4-5]
			\arrow["{\gamma^1}"', shift right=2, from=4-3, to=4-5]
			\arrow["{\gamma_1}", shift left=2, from=4-3, to=6-3]
			\arrow["{\delta_1}"', shift right=2, from=4-3, to=6-3]
			\arrow["{1_{F^0_0}}", shift left=2, from=6-3, to=6-5]
			\arrow["{1_{F^0_0}}"', shift right=2, from=6-3, to=6-5]
			\arrow["{\gamma_0}", shift left=2, from=4-5, to=6-5]
			\arrow["{\delta_0}"', shift right=2, from=4-5, to=6-5]
			\arrow["{\varepsilon^1}"{description}, from=4-5, to=4-3]
			\arrow["{\varepsilon_1}"{description}, from=6-3, to=4-3]
			\arrow["{1_{F^0_0}}"{description}, from=6-5, to=6-3]
			\arrow["{\varepsilon_0}"{description}, from=6-5, to=4-5]
			\arrow["{(\eta_\mathbb{C})^1_1}"', very near end, from=1-1, to=4-3]
			\arrow["{(\eta_\mathbb{C})^1_0}", from=1-3, to=4-5]
			\arrow["{(\eta_\mathbb{C})^0_1}"', from=3-1, to=6-3]
			\arrow["{(\eta_\mathbb{C})^0_0}",  very near start, from=3-3, to=6-5]
		\end{tikzcd}
	\end{equation}
	Here $ (F^0_0,(\eta_\mathbb{C})^0_0) $ is the coequalizer of $ d^0 $ and $ c^0 $. We set $ (\eta_\mathbb{C})^0_1:=(\eta_\mathbb{C})^0_0d^0=(\eta_\mathbb{C})^0_0c^0 $. $ (F^1_0, (\eta_\mathbb{C})^1_0, \varepsilon_0) $ is the pushout of $ (\eta_\mathbb{C})^0_0 $ along $ e_0 $. The morphisms $ \delta_0 $ and $ \gamma_0 $ are the unique ones such that $ \delta_0(\eta_\mathbb{C})^1_0=(\eta_\mathbb{C})^0_0d_0 $ and $ \delta_0\varepsilon_0=1_{F^0_0} $, and $ \gamma_0(\eta_\mathbb{C})^1_0=(\eta_\mathbb{C})^0_0c_0 $ and $ \gamma_0\varepsilon_0=1_{F^0_0} $, respectively, which are induced by the universal property of the pushout $ F^1_0 $. Similarly, $ (F^1_1,(\eta_\mathbb{C})^1_1,\varepsilon_1) $ is the pushout of $ (\eta_\mathbb{C})^0_1 $ along $ e_1 $.  The morphism $ \varepsilon^1 $ is the unique one such that $ \varepsilon^1(\eta_\mathbb{C})^1_0 =(\eta_\mathbb{C})^1_1  e^1 $ and $ \varepsilon^1\varepsilon_0 = \varepsilon_1 $, induced by the universal property of the pushout $ F^1_0 $. The morphisms $ \delta^1 $ and $ \gamma^1 $ are the unique ones such that $ \delta^1(\eta_\mathbb{C})^1_1=(\eta_\mathbb{C})^1_0d^1 $ and $ \delta^1\varepsilon_1=\varepsilon_0 $, and $ \gamma^1(\eta_\mathbb{C})^1_1=(\eta_\mathbb{C})^1_0c^1 $ and $ \gamma^1\varepsilon_1=\varepsilon_0 $, respectively, induced by the universal property of the pushout defining $ F^1_1 $. We set $ \delta_1:=\delta_0\delta^1 $ and $ \gamma_1:=\gamma_0\gamma^1 $.
	
	It is easily seen that the front part of the above diagram is a double reflexive graph $ \mathbb{F} $ in $ \mathscr{C} $. Since $\mathrm{Grpd}(\mathscr{C}) $ is closed under regular quotients in $ \mathrm{RG}(\mathscr{C}) $ when $ \mathscr{C} $ is a regular Mal'tsev category \cite{gran:1999}, it follows that $ \mathbb{F} $ is actually a double groupoid in $ \mathscr{C} $. Furthermore, one checks that $ \eta_\mathbb{C}: \mathbb{C}\to \mathbb{F} $ is a morphism of double reflexive graphs, hence of double groupoids, since $ \mathrm{Grpd}(\mathscr{C}) $ is a full subcategory of $ \mathrm{RG}(\mathscr{C}) $ when $ \mathscr{C} $ is a Mal'tsev category.
	
	Let now 
	\begin{equation*}
		\begin{tikzcd}
			{C^1_1} && {C^1_0} \\
			\\
			{C^0_1} && {C^0_0} \\
			&& {D^1_1} && {D^1_0} \\
			\\
			&& {D^0_1} && {D^0_0} 
			\arrow["{d^1}", shift left=2, from=1-1, to=1-3]
			\arrow["{c^1}"', shift right=2, from=1-1, to=1-3]
			\arrow["{e^1}"{description}, from=1-3, to=1-1]
			\arrow["{c_1}", shift left=2, from=1-1, to=3-1]
			\arrow["{d_1}"', shift right=2, from=1-1, to=3-1]
			\arrow["{e_1}"{description}, from=3-1, to=1-1]
			\arrow["{d^0}", shift left=2, from=3-1, to=3-3]
			\arrow["{c^0}"', shift right=2, from=3-1, to=3-3]
			\arrow["{e^0}"{description}, from=3-3, to=3-1]
			\arrow["{c_0}", shift left=2, from=1-3, to=3-3]
			\arrow["{d_0}"', shift right=2, from=1-3, to=3-3]
			\arrow["{e_0}"{description}, from=3-3, to=1-3]
			\arrow["{d^1}", shift left=2, from=4-3, to=4-5]
			\arrow["{c^1}"', shift right=2, from=4-3, to=4-5]
			\arrow["{d_1}", shift left=2, from=4-3, to=6-3]
			\arrow["{d_1}"', shift right=2, from=4-3, to=6-3]
			\arrow["{d^0}", shift left=2, from=6-3, to=6-5]
			\arrow["{c^0}"', shift right=2, from=6-3, to=6-5]
			\arrow["{c_0}", shift left=2, from=4-5, to=6-5]
			\arrow["{d_0}"', shift right=2, from=4-5, to=6-5]
			\arrow["{e^1}"{description}, from=4-5, to=4-3]
			\arrow["{e_1}"{description}, from=6-3, to=4-3]
			\arrow["{e^0}"{description}, from=6-5, to=6-3]
			\arrow["{e_0}"{description}, from=6-5, to=4-5]
			\arrow["{f^1_1}", very near end, from=1-1, to=4-3]
			\arrow["{f^1_0}"', from=1-3, to=4-5]
			\arrow["{f^0_1}", from=3-1, to=6-3]
			\arrow["{f^0_0}"',  very near start, from=3-3, to=6-5]
		\end{tikzcd}
	\end{equation*}
	represent a double functor $ f:\mathbb{C}\to\mathbb{D} $ between double groupoids $ \mathbb{C} $ and $ \mathbb{D} $. We define $ \mathsf{F}(f) $ as $ \varphi $ in the following diagram: 
	\begin{equation*}
		\begin{tikzcd}
			{F^1_1} && {F^1_0} \\
			\\
			{F^0_0} && {F^0_0} \\
			&& {G^1_1} && {G^1_0} \\
			\\
			&& {G^0_0} && {G^0_0} 
			\arrow["{\delta^1}", shift left=2, from=1-1, to=1-3]
			\arrow["{\gamma^1}"', shift right=2, from=1-1, to=1-3]
			\arrow["{\varepsilon^1}"{description}, from=1-3, to=1-1]
			\arrow["{\gamma_1}", shift left=2, from=1-1, to=3-1]
			\arrow["{\delta_1}"', shift right=2, from=1-1, to=3-1]
			\arrow["{\varepsilon_1}"{description}, from=3-1, to=1-1]
			\arrow["{1_{F^0_0}}", shift left=2, from=3-1, to=3-3]
			\arrow["{1_{F^0_0}}"', shift right=2, from=3-1, to=3-3]
			\arrow["{1_{F^0_0}}"{description}, from=3-3, to=3-1]
			\arrow["{\gamma_0}", shift left=2, from=1-3, to=3-3]
			\arrow["{\delta_0}"', shift right=2, from=1-3, to=3-3]
			\arrow["{\varepsilon_0}"{description}, from=3-3, to=1-3]
			\arrow["{\delta^1}", shift left=2, from=4-3, to=4-5]
			\arrow["{\gamma^1}"', shift right=2, from=4-3, to=4-5]
			\arrow["{\delta_1}", shift left=2, from=4-3, to=6-3]
			\arrow["{\delta_1}"', shift right=2, from=4-3, to=6-3]
			\arrow["{1_{G^0_0}}", shift left=2, from=6-3, to=6-5]
			\arrow["{1_{G^0_0}}"', shift right=2, from=6-3, to=6-5]
			\arrow["{\gamma_0}", shift left=2, from=4-5, to=6-5]
			\arrow["{\delta_0}"', shift right=2, from=4-5, to=6-5]
			\arrow["{\varepsilon^1}"{description}, from=4-5, to=4-3]
			\arrow["{\varepsilon_1}"{description}, from=6-3, to=4-3]
			\arrow["{1_{G^0_0}}"{description}, from=6-5, to=6-3]
			\arrow["{\varepsilon_0}"{description}, from=6-5, to=4-5]
			\arrow["{\varphi^1_1}"', very near end, from=1-1, to=4-3]
			\arrow["{\varphi^1_0}", from=1-3, to=4-5]
			\arrow["{\varphi^0_1}"', from=3-1, to=6-3]
			\arrow["{\varphi^0_0}",  very near start, from=3-3, to=6-5]
		\end{tikzcd}
	\end{equation*}
	Here $ \mathbb{F} $ and $ \mathbb{G} $ are $ \mathsf{F}(\mathbb{C}) $ and $ \mathsf{F}(\mathbb{D})$, respectively, that are both constructed as explained above. We define $ \varphi^0_0 $ to be the unique morphism such that $ \varphi^0_0(\eta_\mathbb{C})^0_0=(\eta_\mathbb{D})^0_0 f^0_0 $ induced by the universal property of the coequalizer $ (\eta_\mathbb{C})^0_0 \colon C^0_0 \rightarrow F^0_0 $. We set $ \varphi^0_1:=\varphi^0_0 $. We define $ \varphi^1_0 $ to be the unique morphism such that $ \varphi^1_0\varepsilon_0=\varepsilon_0\varphi^0_0 $ and $ \varphi^1_0(\eta_\mathbb{C})^1_0=(\eta_\mathbb{D})^1_0 f^1_0 $ induced by the universal property of the pushout defining $ G^1_0 $. Finally, we define $\varphi^1_1 $ to be the unique morphism such that $ \varphi^1_1 \varepsilon_1=\varepsilon_1 \varphi^0_1 $ and $ \varphi^1_1 (\eta_\mathbb{C})^1_1=(\eta_\mathbb{D})^1_1 f^1_1 $ induced by the universal property of the pushout defining $ G^1_1 $. 
	
	It is easily seen that $ \varphi $ is a double functor from $ \mathbb{F} $ to $ \mathbb{G} $. Moreover, $ \eta $ is a natural transformation from $ 1_{\mathrm{Grpd}^2(\mathscr{C})} $ to $ \mathsf{UF} $, where $ \mathsf{U} $ is the inclusion functor of $ 2\textrm{-Grpd}(\mathscr{C}) $ into $ \mathrm{Grpd}^2(\mathscr{C}) $. Note that $ \eta_\mathbb{C} $ is a regular epimorphism in $ \mathrm{Grpd}^2(\mathscr{C}) $ since $ (\eta_\mathbb{C})^i_j $ is a regular epimorphism in $ \mathscr{C} $ for all $ i,j\in\{0,1\} $. 
	
	Let now $ f:\mathbb{C}\to\mathbb{D} $ be a morphism in $ \mathrm{Grpd}^2(\mathscr{C})$, where $ \mathbb{D}\in2\text{-}\mathrm{Grpd}(\mathscr{C}) $, as depicted in the left-hand side of the following diagram: 
	
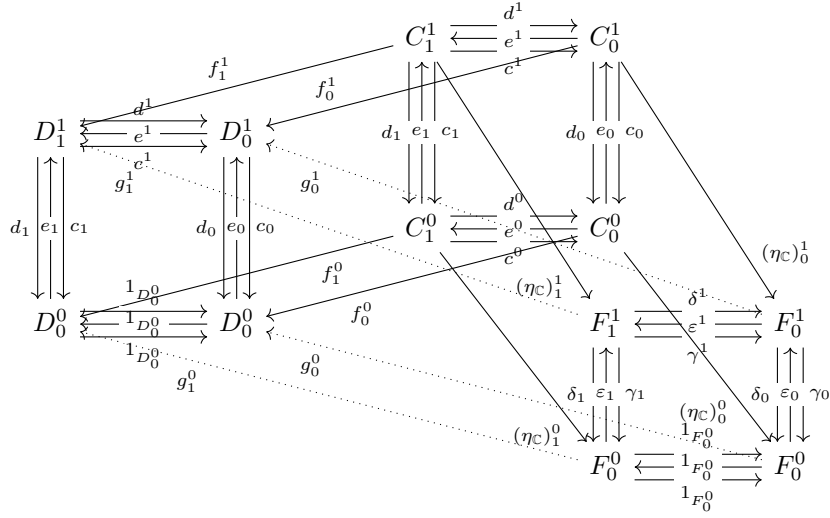
\begin{equation*}
		\begin{tikzcd}
			&&&& {C^1_1} && {C^1_0} \\
			{D^1_1} && {D^1_0} \\
			&&&& {C^0_1} && {C^0_0} \\
			{D^0_0} && {D^0_0} &&&& {F^1_1} && {F^1_0} \\
			\\
			&&&&&& {F^0_0} && {F^0_0}
			\arrow["{d^1}", shift left=2, from=1-5, to=1-7]
			\arrow["{c^1}"', shift right=2, from=1-5, to=1-7]
			\arrow["{e^1}"{description}, from=1-7, to=1-5]
			\arrow["{c_1}", shift left=2, from=1-5, to=3-5]
			\arrow["{d_1}"', shift right=2, from=1-5, to=3-5]
			\arrow["{e_1}"{description}, from=3-5, to=1-5]
			\arrow["{d^0}", shift left=2, from=3-5, to=3-7]
			\arrow["{c^0}"', shift right=2, from=3-5, to=3-7]
			\arrow["{e^0}"{description}, from=3-7, to=3-5]
			\arrow["{c_0}", shift left=2, from=1-7, to=3-7]
			\arrow["{d_0}"', shift right=2, from=1-7, to=3-7]
			\arrow["{e_0}"{description}, from=3-7, to=1-7]
			\arrow["{\delta^1}", shift left=2, from=4-7, to=4-9]
			\arrow["{\gamma^1}"', shift right=2, from=4-7, to=4-9]
			\arrow["{\gamma_1}", shift left=2, from=4-7, to=6-7]
			\arrow["{\delta_1}"', shift right=2, from=4-7, to=6-7]
			\arrow["{1_{F^0_0}}", shift left=2, from=6-7, to=6-9]
			\arrow["{1_{F^0_0}}"', shift right=2, from=6-7, to=6-9]
			\arrow["{\gamma_0}", shift left=2, from=4-9, to=6-9]
			\arrow["{\delta_0}"', shift right=2, from=4-9, to=6-9]
			\arrow["{\varepsilon^1}"{description}, from=4-9, to=4-7]
			\arrow["{\varepsilon_1}"{description}, from=6-7, to=4-7]
			\arrow["{1_{F^0_0}}"{description}, from=6-9, to=6-7]
			\arrow["{\varepsilon_0}"{description}, from=6-9, to=4-9]
			\arrow["{(\eta_\mathbb{C})^1_1}"', very near end, from=1-5, to=4-7]
			\arrow["{(\eta_\mathbb{C})^1_0}", very near end, from=1-7, to=4-9]
			\arrow["{(\eta_\mathbb{C})^0_1}"',  very near end, from=3-5, to=6-7]
			\arrow["{(\eta_\mathbb{C})^0_0}"', near end, from=3-7, to=6-9]
			\arrow["{f^1_1}"', from=1-5, to=2-1]
			\arrow[dotted, "{g^1_1}", very near end, from=4-7, to=2-1]
			\arrow["{f^1_0}"', near end, from=1-7, to=2-3]
			\arrow[dotted, "{ g^1_0}", very near end, from=4-9, to=2-3]
			\arrow["{f^0_1}", near start, from=3-5, to=4-1]
			\arrow["{f^0_0}", near end, from=3-7, to=4-3]
			\arrow[dotted,"{g^0_1}", near end, from=6-7, to=4-1]
			\arrow[dotted, "{g^0_0}", very near end, from=6-9, to=4-3]
			\arrow["{e_0}"{description}, from=4-3, to=2-3]
			\arrow["{e^1}"{description}, from=2-3, to=2-1]
			\arrow["{e_1}"{description}, from=4-1, to=2-1]
			\arrow["{1_{D^0_0}}"{description}, from=4-3, to=4-1]
			\arrow["{c_1}", shift left=2, from=2-1, to=4-1]
			\arrow["{d_1}"', shift right=2, from=2-1, to=4-1]
			\arrow["{d^1}", shift left=2, from=2-1, to=2-3]
			\arrow["{c^1}"', shift right=2, from=2-1, to=2-3]
			\arrow["{c_0}", shift left=2, from=2-3, to=4-3]
			\arrow["{d_0}"', shift right=2, from=2-3, to=4-3]
			\arrow["{1_{D^0_0}}", shift left=2, from=4-1, to=4-3]
			\arrow["{1_{D^0_0}}"', shift right=2, from=4-1, to=4-3]
		\end{tikzcd}
	\end{equation*}
	We show that there exists a unique double functor $ g:\mathbb{F}\to\mathbb{D} $ such that $ g\eta_\mathbb{C}=f $. Since $ f^0_0d^0=f^0_1=f^0_0c^0 $, there exists a unique morphism $ g^0_0 $ such $ g^0_0(\eta_\mathbb{C})^0_0=f^0_0 $ by the universal property of the coequalizer $ F^0_0 $. We set $ g^0_1:=g^0_0 $. Since $ e_0g^0_0(\eta_\mathbb{C})^0_0=e_0f^0_0=f^1_0e_0 $, there exists a unique morphism $ g^1_0 $ such that $ g^1_0(\eta_\mathbb{C})^1_0=f^1_0 $ and $ g^1_0\varepsilon_0=e_0g^0_0 $ by the universal property of the pushout $ F^1_0 $. Since $$ e_1g^0_1(\eta_\mathbb{C})^0_1=e_1g^0_0(\eta_\mathbb{C})^0_0d^0=e_1f^0_0d^0=e_1f^0_1=f^1_1e_1, $$ by the universal property of the pushout $ F^1_1 $ there exists a unique morphism $ g^1_1 \colon F^1_1 \rightarrow D^1_1$ such that $ g^1_1(\eta_\mathbb{C})^1_1=f^1_1 $ and $g^1_1 \varepsilon_1 = e_1 g^0_1$. It is easily seen that the constructed $ g $ is an internal double functor. Its uniqueness follows from the fact that $ \eta_\mathbb{C} $ is a regular epimorphism. 
	
	To complete the proof, it suffices to observe that 2-$\mathrm{Grpd}(\mathscr{C}) $ is closed under regular quotients in $ \mathrm{Grpd}^2(\mathscr{C}) $.
	Given any regular epimorphism $f \colon \mathbb{C} \rightarrow \mathbb D$ in $ \mathrm{Grpd}^2(\mathscr{C})$ 
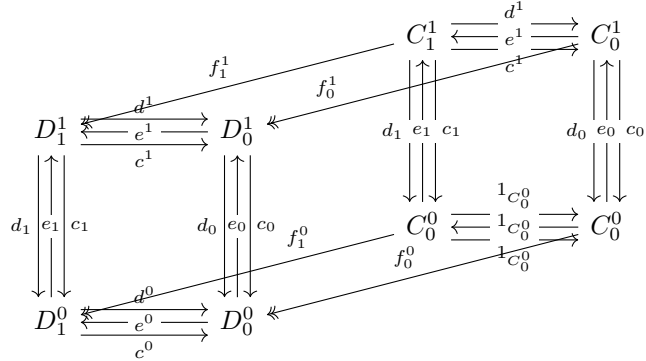
	\begin{equation*}
		\begin{tikzcd}
			&&&& {C^1_1} && {C^1_0} \\
			{D^1_1} && {D^1_0} \\
			&&&& {C^0_0} && {C^0_0} \\
			{D^0_1} && {D^0_0}
			\arrow["{d^1}", shift left=2, from=1-5, to=1-7]
			\arrow["{c^1}"', shift right=2, from=1-5, to=1-7]
			\arrow["{e^1}"{description}, from=1-7, to=1-5]
			\arrow["{c_1}", shift left=2, from=1-5, to=3-5]
			\arrow["{d_1}"', shift right=2, from=1-5, to=3-5]
			\arrow["{e_1}"{description}, from=3-5, to=1-5]
			\arrow["{1_{C^0_0}}", shift left=2, from=3-5, to=3-7]
			\arrow["{1_{C^0_0}}"', shift right=2, from=3-5, to=3-7]
			\arrow["{1_{C^0_0}}"{description}, from=3-7, to=3-5]
			\arrow["{c_0}", shift left=2, from=1-7, to=3-7]
			\arrow["{d_0}"', shift right=2, from=1-7, to=3-7]
			\arrow["{e_0}"{description}, from=3-7, to=1-7]
			\arrow["{f^1_1}"', two heads, from=1-5, to=2-1]
			\arrow["{f^1_0}"', near end, two heads, from=1-7, to=2-3]
			\arrow["{f^0_1}"',near start, two heads, from=3-5, to=4-1]
			\arrow["{f^0_0}"', two heads, from=3-7, to=4-3]
			\arrow["{e_0}"{description}, from=4-3, to=2-3]
			\arrow["{e^1}"{description}, from=2-3, to=2-1]
			\arrow["{e_1}"{description}, from=4-1, to=2-1]
			\arrow["{e^0}"{description}, from=4-3, to=4-1]
			\arrow["{c_1}", shift left=2, from=2-1, to=4-1]
			\arrow["{d_1}"', shift right=2, from=2-1, to=4-1]
			\arrow["{d^1}", shift left=2, from=2-1, to=2-3]
			\arrow["{c^1}"', shift right=2, from=2-1, to=2-3]
			\arrow["{c_0}", shift left=2, from=2-3, to=4-3]
			\arrow["{d_0}"', shift right=2, from=2-3, to=4-3]
			\arrow["{d^0}", shift left=2, from=4-1, to=4-3]
			\arrow["{c^0}"', shift right=2, from=4-1, to=4-3]
		\end{tikzcd}
	\end{equation*}
	where $\mathbb C$
	is a $2$-groupoid, the fact that $f^0_1 $ is a regular epimorphism in $\mathscr{C}$ easily implies that $e^0$ is an isomorphism, so that $ \mathbb{D}\in2\text{-}\mathrm{Grpd}(\mathscr{C}) $, as desired. 
\end{proof}

\begin{remark}\label{rem:alternativeproof}
	More generally, one can consider the following situation: let $ \mathscr{D} $ be a regular Mal'tsev category with finite colimits, and let $ \mathscr{X} $ be a Birkhoff subcategory of $ \mathscr{D} $. We define $ \overline{\mathscr{X}} $ to be the full subcategory of $ \mathrm{Grpd}(\mathscr{D}) $ with objects the internal groupoids $ \mathbb{C} $ whose ``object of objects" $ C_0 $ lies in $ \mathscr{X} $. By using the stability under regular quotients of the subcategory of groupoids in the category of reflexive graphs \cite{gran:1999} one can show that $ \overline{\mathscr{X}} $ is a Birkhoff subcategory of $ \mathrm{Grpd}(\mathscr{D}) $. In particular, this can be applied to the situation considered in Theorem \ref{Birkhoff}, where $ \mathscr{D}=\mathrm{Grpd}(\mathscr{C}) $ and $ \mathscr{X} $ is the Birkhoff subcategory of $ \mathrm{Grpd}(\mathscr{C}) $ whose objects are the discrete groupoids in $ \mathscr{C} $. In this case $ \overline{\mathscr{X}}\approx 2\textrm{-}\mathrm{Grpd}(\mathscr{C}) $ is a Birkhoff subcategory of $ \mathrm{Grpd}(\mathscr{D}) = \mathrm{Grpd}^2(\mathscr{C})$.
\end{remark}

\begin{corollary}
	Let $ \mathscr{C} $ be a finitely cocomplete regular  Mal'tsev category. Then $ 2\textrm{-Grpd}(\mathscr{C}) $ is also a finitely cocomplete regular Mal'tsev category. Moreover, when $\mathscr{C}$ is also exact, then the category $ 2\textrm{-Grpd}(\mathscr{C}) $ is also exact.
\end{corollary}
\begin{proof}
	As recalled before, a regular epi-reflective subcategory $ \mathscr{X} $ of a regular  Mal'tsev category $ \mathscr{C} $ is again a regular  Mal'tsev category.  If $ \mathscr{C} $ is a finitely cocomplete regular Mal'tsev category, then $ \mathrm{Grpd}^2(\mathscr{C}) $ satisfies the same properties \cite{gran:1999}. Hence by Theorem~\ref{Birkhoff} it follows that the same holds for $ 2\textrm{-Grpd}(\mathscr{C}) $.
	Finally, when  $\mathscr{C} $ is exact, the Birkhoff assumption implies that 2-$ \mathrm{Grpd}(\mathscr{C}) $ is exact. 
\end{proof}

It is well-known that the category $ \mathrm{Grpd}(\mathscr{C}) $ of internal groupoids in a Mal'tsev variety $\mathscr{C}$ can be presented as a finitary variety of universal algebras, since it is a Birkhoff subcategory of the variety $\mathrm{RG}(\mathscr{C})$ of reflexive graphs in $\mathscr{C}$ (see \cite{janelidze:1990}, and Corollary $2.4$ in \cite{Gran-Rosicky}, which is based on some results in \cite{Ja-Pe}).
There is a simple relationship between the algebraic theory of the Mal'tsev variety $\mathscr{C}$ and the theory of the variety 
$ \mathrm{Grpd}(\mathscr{C}) $.
Indeed, let us assume that $\mathscr{C}$ is a Mal'tsev variety whose theory contains a Mal'tsev term $p(x,y,z)$, some $n_i$-ary terms ${\omega}_i (x_1, \cdots, x_{n_i})$ for $i \in I$ and $n_i \in {\mathbb N}$, satisfying some identities $\tau_j = \sigma_j$, for some terms $\tau_j$ and $\sigma_j$ with $j \in J$. Then the category $ \mathrm{Grpd}(\mathscr{C}) $ is equivalent to the Mal'tsev variety whose algebraic theory contains all the terms and the identities above, plus two additional unary operations $s$ and $t$, such that $s$ and $t$ are homomorphisms in $\mathscr{C}$, $st=t$ and $ts=s$, and the universal algebraic commutator $[Eq(s), Eq(t)]$ of the kernel congruences $Eq(s)$ and $Eq(t)$ of $s$ and $t$ is trivial. This algebraic description of the category of internal groupoids is an extension to Mal'tsev varieties of the one originally given by J.-L. Loday in \cite{Loday} in the special case of the variety of groups. Starting from a groupoid whose underlying reflexive graph is diagram \eqref{reflexive-graph}, the unary operations $s$ and $t$ are the composite homomorphisms $s=ed$ and $t=ec$, that clearly satisfy the identities $st=t$ and $ts=s$. The commutator $[Eq(s), Eq(t)]$ vanishes if and only if there is a unique groupoid structure on the reflexive graph \eqref{reflexive-graph} \cite{Ja-Pe}.

\begin{corollary}\label{theory-variety}
	Let $ \mathscr{C} $ be a Mal'tsev variety of universal algebras. Then $ 2\text{-}\mathrm{Grpd}(\mathscr{C}) $ is a Mal'tsev variety.
\end{corollary}
\begin{proof}
	The category $\mathrm{Grpd}^2(\mathscr{C}) $ of double groupoids in $\mathscr{C}$ is also a variety, since it is again a category of internal groupoids in the Mal'tsev variety $ \mathrm{Grpd}(\mathscr{C}) $. More explicitly, when $\mathscr{C}$ is a Mal'tsev variety as above, then $\mathrm{Grpd}^2(\mathscr{C}) $ is equivalent to a Mal'tsev variety whose algebraic theory contains all the terms $p(x,y,z)$, ${\omega}_i (x_1, \cdots, x_{n_i})$ for $i \in I$ and $n_i \in {\mathbb N}$ satisfying the identities $\tau_j = \sigma_j$, for some terms $\tau_j$ and $\sigma_j$ with $j \in J$,
	together with four additional unary operations $s$, $t$, $u$ and $v$ satisfying the following conditions:
	\begin{itemize}
		\item $st = t, ts = s, uv = v, vu = u, su = us, sv = vs, tu = ut, tv = vt$;
		\item the unary operations $s$, $t$, $u$ and $v$ are algebra homomorphisms in $\mathscr{C}$;
		\item the commutators $[Eq(s), Eq(t)]$ and $[Eq(u), Eq(v)]$ are trivial.   
	\end{itemize}
	By Theorem \ref{Birkhoff} we deduce that $ 2\text{-}\mathrm{Grpd}(\mathscr{C}) $ is a subvariety of the variety described here above. Indeed, with this presentation of the algebraic theory of $ \mathrm{Grpd}^2(\mathscr{C}) $ the additional identities determining the subvariety 
	$ 2\text{-}\mathrm{Grpd}(\mathscr{C}) $ of $2$-groupoids in $\mathscr{C} $  are the following:
	$$u = us = ut, v = vs = vt.$$
\end{proof}
\begin{remark}
	The stability under regular quotients of the category $\mathrm{Grpd}(\mathscr{C})$ in $\mathrm{RG}(\mathscr{C})$ \cite{GRT} also holds in any Goursat category $ \mathscr{C} $ in the sense of \cite{CKP}.
	This implies that Theorem \ref{Birkhoff} still holds when $\mathscr{C}$ is a Goursat category with finite colimits. However, when $\mathscr{C}$ is a Goursat (=$3$-permutable) variety, it is no longer true that  $\mathrm{Grpd}(\mathscr{C})$ is a subvariety of $\mathrm{RG}(\mathscr{C})$ since it is not stable under subobjects \cite{Gran-Rosicky}, hence Corollary \ref{theory-variety} does not apply.
	
\end{remark}

\section{Naturally Mal'tsev categories}

Recall that a category $\mathscr{C}$ with products is a naturally Mal'tsev category if there is a natural transformation $p \colon \mathsf{Id}_\mathscr{C}  \times \mathsf{Id}_\mathscr{C} \times \mathsf{Id}_\mathscr{C} \rightarrow \mathsf{Id}_\mathscr{C}$ such that each component $p_A \colon A \times A \times A \rightarrow A$ of this natural transformation is an internal Mal'tsev operation for any $A \in \mathscr{C}$ \cite{Johnstone}. Naturally Mal'tsev categories are characterized by the property that the forgetful functor from the category of internal groupoids to the category of reflexive graphs is an isomorphism. This also implies that the category of double groupoids is isomorphic to the category of double reflexive graphs.

We are going to show that in the naturally Mal'tsev context the adjunction between the categories of internal $2$-groupoids and of double groupoids described in the previous section
has a remarkable additional property.
For this, we shall use a characterization of naturally Mal'tsev categories due to D. Bourn (see Proposition 7.8.1 in \cite{BB}) that we shall reformulate using the following observation.

Given a regular epimorphism $q:X\to Y$ and an arbitrary morphism $\delta:X\to S$ in a category $\mathscr{C}$ with kernel pairs and coequalizers, the pushout of $q$ along $\delta$ exists and can be constructed in the following way. We consider the kernel pair $(\mathrm{Eq}(q),q_1,q_2)$ of $q$ and construct the coequalizer $(C,c)$ of $\delta q_1$ and $\delta q_2$. This yields a unique morphism $\bar{\delta}$ such that $\bar{\delta}q=c\delta$. Then the commutative square 
\begin{equation}\label{pushout}
	\begin{tikzcd}
		S \arrow[r, "c"] & C\\
		X \arrow[u, "\delta"] \arrow[r, "q"'] & Y \arrow[u, "\bar{\delta}"']
	\end{tikzcd}  
\end{equation}
is easily seen to be the pushout of $q$ and $\delta$.

If $ \delta $ is a split monomorphism, so that there is an $ f:S\to X $ such that $ f\delta=1_X $, there is a unique morphism $ \bar{f} $ such that $ \bar{f}\bar{\delta}=1_Y $ and $ \bar{f}c=q\delta $, i.e., $ \bar{\delta} $ is a split monomorphism with splitting $ \bar{f} $ and the diagram 
\begin{equation}\label{pullback}
	\begin{tikzcd}
		S \arrow[d, "f"'] \arrow[r, "c"] & C \arrow[d, "\overline{f}"] \\
		X \arrow[r, "q"'] & Y
	\end{tikzcd}
\end{equation}
is commutative. 

\begin{proposition}\label{BB} \cite{BB}
	Let $\mathscr{C}$ be an exact Mal'tsev category with coequalizers. The following conditions are equivalent: 
	\begin{enumerate}
		\item $\mathscr{C}$ is a naturally Mal'tsev category; 
		\item given a regular epimorphism $q \colon X \rightarrow Y$ and a split monomorphism $\delta \colon X \rightarrow S$ with splitting $ f $, the induced commutative square \eqref{pullback}
		is a pullback. 
	\end{enumerate}
\end{proposition}

\begin{proposition}\label{units-NM}
	Let $\mathscr{C}$ be an exact naturally Mal'tsev category with coequalizers. Given a double groupoid $\mathbb C$ in $\mathscr{C}$, depicted as \begin{equation*}
		\begin{tikzcd}
			C^1_1 \arrow[d, shift right=3, "d_1"'] \arrow[d, shift left=3, "c_1"] \arrow[r, shift left=3, "d^1"] \arrow[r, shift right=3, "c^1"']  &C^1_0 \arrow[d, shift right=3, "d_0"'] \arrow[d, shift left=3, "c_0"] \arrow[l, "e^1" description] \\
			C^0_1 \arrow[u, "e_1" description] \arrow[r, shift left=3, "d^0"] \arrow[r, shift right=3, "c^0"'] &{C^0_0,} \arrow[u, "e_0" description] \arrow[l, "e^0" description]
		\end{tikzcd}
	\end{equation*} 
	by applying the  reflector $F \colon \mathrm{Grpd}^2(\mathscr{C}) \rightarrow 2\text{-}\mathrm{Grpd}(\mathscr{C})$ to $\mathbb C$ 
	the internal functors
	\begin{equation*}
		\begin{tikzcd}
			C_0^1 \arrow[r, "{(\eta_C)}_0^1"] \arrow[d, shift left=3, "c_0"] \arrow[d, shift right=3, "d_0"'] &F_1^0 \arrow[d, shift left=3, "\gamma_0"] \arrow[d, shift right=3, "\delta_0"']   & \\
			C_0^0 \arrow[u, "e_0" description] \arrow[r, "{(\eta_C)}_0^0"'] &F_0^0, \arrow[u, "\varepsilon_0" description] &
		\end{tikzcd}
		\begin{tikzcd}
			& C_1^1 \arrow[r, "{(\eta_C)}_1^1"] \arrow[d, shift left=3, "c_1"] \arrow[d, shift right=3, "d_1"'] &F_1^0 \arrow[d, shift left=3, "\gamma_1"] \arrow[d, shift right=3, "\delta_1"']  \\
			&  F_1^1 \arrow[u, "e_1" description] \arrow[r, "{(\eta_C)}_1^0"']  &F_1^0, \arrow[u, "\varepsilon_1" description]
		\end{tikzcd}
	\end{equation*}
	in  diagram \eqref{reflection} are both discrete fibrations: all the downward directed commutative squares are pullbacks.
\end{proposition}
\begin{proof}
	This is a direct consequence of the construction of the reflector $F \colon \mathrm{Grpd}^2(\mathscr{C}) \rightarrow 2\text{-}\mathrm{Grpd}(\mathscr{C})$ 
	and Proposition \ref{BB}. 
\end{proof}

\begin{remark}
	By looking at the construction given by M.C. Pedicchio in \cite{Pedicchio} of the commutator of two equivalence relations $R$ and $S$ in an exact Mal'tsev category with coequalizers, one sees that the construction of the smallest double equivalence relation $\Delta_R^S$ on $R$ and $S$ uses the same type of pushouts \eqref{pushout} that we have considered here above. Given $R$ and $S$, one first takes the coequalizer $q \colon X \rightarrow X/R$ and then the coequalizer $c \colon S\rightarrow S/R$ of $\delta_S r_1$ and 
	$\delta_S r_2$:
	\begin{equation*}
		\begin{tikzcd}
			\Delta_R^S \arrow[d, shift right=3, "p_1"'] \arrow[d, shift left=3, "p_2"] \arrow[r, shift left=3, "\pi_1"] \arrow[r, shift right=3, "\pi_2"'] & S \arrow[l, "\delta_\pi"description] \arrow[d, shift right=3, "s_1"'] \arrow[d, shift left=3, "s_2"] \arrow[r, "c"] & S/R \arrow[d, shift right=3, "\bar{s_1}"'] \arrow[d, shift left=3, "\bar{s_2}"]\\
			R \arrow[u, "\delta_p"description] \arrow[r, shift left=3, "r_1"] \arrow[r, shift right=3, "r_2"'] & X \arrow[l, "\delta_R"description] \arrow[u, "\delta_S"description] \arrow[r, "q"'] & X/R \arrow[u, "\bar{\delta_S}"description]
		\end{tikzcd}
	\end{equation*}
	Note that, as seen above, $(S/R, c,\bar{\delta_S})$ is also the pushout of $q$ and $\delta_S$.
	The double equivalence relation $\Delta_R^S$ on $R$ and $S$ is obtained by taking the kernel pair $(\Delta_R^S, \pi_1, \pi_2)$ of $c$, and then it is easy to see that the induced maps $p_1$, $p_2$ and $\delta_p$ determine a reflexive relation, hence an equivalence relation, on $R$. 
	Thanks to Proposition \ref{BB}, when $\mathscr{C}$ is an exact naturally Mal'tsev category with coequalizers, this double equivalence relation is necessarily a \emph{double centralizing relation} on $R$ and $S$, which means that all the commutative squares on the left are pullbacks. This implies that the categorical commutator $[R,S]$ of $R$ and $S$ is trivial, i.e. the smallest equivalence relation on $X$. This shows an interesting and unexpected connection between commutator theory and the reflector $F$ from the category of internal double groupoids to the one of $2$-groupoids considered in the previous section.
\end{remark}

\section{Action representability of the category $ 2\textrm{-Grpd}(\mathscr{C}) $ }\label{Action-R}

The starting point of this section is the following

\begin{theorem}\label{thm:Grpd(C)actionrepresentable}\cite[Theorem 2.10]{gran.gray:2021}
	For a category $ \mathscr{C} $, the following conditions are equivalent: 
	\begin{enumerate}
		\item $ \mathscr{C} $ is a semi-abelian action representable algebraically coherent category with normalizers;
		\item $ \mathrm{Grpd}(\mathscr{C}) $ is a semi-abelian action representable algebraically coherent category with normalizers. 
	\end{enumerate}
\end{theorem}

We will show in this section that if $ \mathscr{C} $ is  semi-abelian action representable algebraically coherent with normalizers then also $ 2\textrm{-Grpd}(\mathscr{C}) $ is so. Examples of such  categories $ \mathscr{C} $ are given by the categories of groups, of Lie algebras over a commutative ring, and of cocommutative Hopf algebras over a field. 

We begin this section by proving 

\begin{corollary}
	Let $ \mathscr{C} $ be a semi-abelian category. Then $ 2\text{-}\mathrm{Grpd}(\mathscr{C}) $ is a semi-abelian category.
\end{corollary}
\begin{proof}
	If $\mathscr{C}$ is a semi-abelian category, then $ \mathrm{Grpd}(\mathscr{C})$, and then $ \mathrm{Grpd}^2(\mathscr{C})$, are again semi-abelian \cite{Bourn-Gran}. Any Birkhoff subcategory of a semi-abelian category is itself semi-abelian, hence so is the category $ 2\text{-}\mathrm{Grpd}(\mathscr{C}) $, by Theorem \ref{Birkhoff}.
\end{proof}

From now on we shall always assume that $\mathscr{C}$ is semi-abelian. Let us then recall the notion of action representable category, of algebraically coherent category, and of category with normalizers.

A split extension in $ \mathscr{C} $ is a diagram 
\begin{equation}\label{split-extension}
	\begin{tikzcd}
		X \arrow[r, "\kappa"] &A \arrow[r, shift left=2, "\alpha"] &B \arrow[l, shift left=2, "\beta"]
	\end{tikzcd}
\end{equation}
in $ \mathscr{C} $, where $ \alpha\beta=1_B $ and $ \kappa $ is the kernel of $ \alpha $. A morphism of split extensions in $ \mathscr{C} $ is a diagram 
\begin{equation*}
	\begin{tikzcd}
		X\arrow[d, "u"'] \arrow[r, "\kappa"] &A \arrow[d, "v"] \arrow[r, shift left=2, "\alpha"] &B \arrow[d, "w"] \arrow[l, shift left=2, "\beta"]\\
		X' \arrow[r, "\kappa'"] &A' \arrow[r, shift left=2, "\alpha'"] &B \arrow[l, shift left=2, "\beta'"]
	\end{tikzcd}
\end{equation*}
in $ \mathscr{C} $, where the top row is the domain split extension, the bottom row is the codomain split extension, $ v\kappa=\kappa'u$, $ \alpha'v=w\alpha $ and $ v\beta=\beta'w$. Let us denote by $ \mathrm{SplExt}(\mathscr{C}) $ the category of split extensions in $ \mathscr{C} $ and by $ P: \mathrm{SplExt}(\mathscr{C})\to\mathscr{C} $ and $ K: \mathrm{SplExt}(\mathscr{C})\to\mathscr{C} $ the functors sending a split extension \eqref{split-extension} to $ B $ and to $ X $, respectively. The category $ \mathscr{C} $ is action representable \cite{BJK} if and only if each fiber of the functor $ K $ has a terminal object. This means that for any $ X $ in $ \mathscr{C} $ there is a split extension 
\begin{equation*}
	\begin{tikzcd}
		X \arrow[r, "k"] &
		\overline{X} \arrow[r, shift left=2, "p_1"] &{[}X{]}, \arrow[l, shift left=2, "i"]
	\end{tikzcd}
\end{equation*}
called the ``generic split extension with kernel $ X $'', such that for any split extension \eqref{split-extension} there is a unique morphism of split extensions
\begin{equation*}
	\begin{tikzcd}
		X \arrow[d, equal] \arrow[r, "\kappa"] &A \arrow[d, "v"] \arrow[r, shift left=2, "\alpha"] &B \arrow[d, "w"] \arrow[l, shift left=2, "\beta"]\\
		X \arrow[r, "k"] & \overline{X} \arrow[r, shift left=2, "p_1"] &{[}X{]}. \arrow[l, shift left=2, "i"]
	\end{tikzcd}
\end{equation*}
When a generic split extension with kernel $X$ as above exists, the object $[X]$ is often called the actor, or the split extension classifier, of $X$.
If $\mathscr{C}$ is the category of groups, the split extension classifier $ [G] $ of a group $G$ is given by its automorphism group $ \mathrm{Aut}(G) $. In the category of Lie algebras over a commutative ring, the  split extension classifier $ [L] $ of a Lie algebra $ L $ is given by its Lie algebra of derivations $ \mathrm{Der}(L) $.

A pointed protomodular category $ \mathscr{C} $ is algebraically coherent \cite{cigoli.gray.vanderlinden:2015} when, for each cospan of monomorphisms of split extensions 
\begin{equation*}
	\begin{tikzcd}
		X_1 \arrow[d, "u_1"'] \arrow[r, "\kappa_1"] &A_1 \arrow[d, "v_1"] \arrow[r, shift left=1, "\alpha_1"] &B \arrow[d, equal] \arrow[l, shift left=2, "\beta_1"]\\
		X \arrow[r, "\kappa"] &A \arrow[r, shift left=1, "\alpha"] &B \arrow[l, shift left=2, "\beta"]\\
		X_2 \arrow[u, "u_2"] \arrow[r, "\kappa_2"] &A_2 \arrow[u, "v_2"'] \arrow[r, shift left=1, "\alpha_2"] &B, \arrow[u, equal] \arrow[l, shift left=2, "\beta_2"]
	\end{tikzcd}
\end{equation*}
if the morphisms $ v_1 $ and $ v_2 $ are jointly strongly epimorphic in $ \mathscr{C} $, then so are the morphisms $ u_1 $ and $ u_2 $. 

A category $ \mathscr{C} $ has normalizers \cite{Gray} if for any monomorphism $ f: A\to B $ in $ \mathscr{C} $, the category with objects triples $ (N,n,m) $ where $ n $ is a normal monomorphism, $ m $ is a monomorphism and $ f=mn $, with expected morphisms, has a terminal object, the so-called normalizer of $ f $. 

We have seen before that $ 2\textrm{-Grpd}(\mathscr{C}) $ is a Birkhoff subcategory of $ \textrm{Grpd}^2(\mathscr{C}) $ whenever $ \mathscr{C} $ is a finitely cocomplete regular Mal'tsev category. In a more general context we have  the following

\begin{lemma}\label{prop:2-Grpd(C)coreflective}
	Let $ \mathscr{C} $ be a finitely complete category. Then $ 2\textrm{-Grpd}(\mathscr{C}) $ is a regular mono-coreflective subcategory of $ \mathrm{Grpd}^2(\mathscr{C}) $. 
\end{lemma}
\begin{proof}
	Let us describe the coreflection $ \mathsf{G}:\mathrm{Grpd}^2(\mathscr{C})\to 2\textrm{-Grpd}(\mathscr{C}) $. For a double groupoid $ \mathbb{C} $ in $ \mathscr{C} $, we define $ \mathsf{G}(\mathbb{C}) $ as $ \mathbb{F} $ in the following diagram: 
	\begin{equation}\label{coreflection}
		\begin{tikzcd}
			{C^1_1} && {C^1_0} \\
			\\
			{C^0_1} && {C^0_0} \\
			&& {F^1_1  } && {F^1_0 = C^1_0} \\
			\\
			&& {F^0_0 =C^0_0} && {F^0_0 =C^0_0.} 
			\arrow["{d^1}", shift left=2, from=1-1, to=1-3]
			\arrow["{c^1}"', shift right=2, from=1-1, to=1-3]
			\arrow["{e^1}"{description}, from=1-3, to=1-1]
			\arrow["{c_1}", shift left=2, from=1-1, to=3-1]
			\arrow["{d_1}"', shift right=2, from=1-1, to=3-1]
			\arrow["{e_1}"{description}, from=3-1, to=1-1]
			\arrow["{d^0}", shift left=2, from=3-1, to=3-3]
			\arrow["{c^0}"', shift right=2, from=3-1, to=3-3]
			\arrow["{e^0}"{description}, from=3-3, to=3-1]
			\arrow["{c_0}", shift left=2, from=1-3, to=3-3]
			\arrow["{d_0}"', shift right=2, from=1-3, to=3-3]
			\arrow["{e_0}"{description}, from=3-3, to=1-3]
			\arrow["{\delta^1}", shift left=2, from=4-3, to=4-5]
			\arrow["{\gamma^1}"', shift right=2, from=4-3, to=4-5]
			\arrow["{\gamma_1}", shift left=2, from=4-3, to=6-3]
			\arrow["{\delta_1}"', shift right=2, from=4-3, to=6-3]
			\arrow["1_{C^0_0}", shift left=2, from=6-3, to=6-5]
			\arrow["1_{C^0_0}"', shift right=2, from=6-3, to=6-5]
			\arrow["{\gamma_0}", shift left=2, from=4-5, to=6-5]
			\arrow["{\delta_0}"', shift right=2, from=4-5, to=6-5]
			\arrow["{\varepsilon^1}"{description}, from=4-5, to=4-3]
			\arrow["{\varepsilon_1}"{description}, from=6-3, to=4-3]
			\arrow["1_{C^0_0}"{description}, from=6-5, to=6-3]
			\arrow["{\varepsilon_0}"{description}, from=6-5, to=4-5]
			\arrow["{(\varepsilon_C)^1_1}", very near start, to=1-1, from=4-3]
			\arrow["{(\varepsilon_C)^1_0= 1_{C_0^1}}"', to=1-3, from=4-5]
			\arrow["{(\varepsilon_C)^0_1 = e^0}", to=3-1, from=6-3]
			\arrow["{(\varepsilon_C)^0_0 = 1_{C^0_0}}"',  very near end, to=3-3, from=6-5]
		\end{tikzcd}
	\end{equation}
	Here we set $ F^0_0:=C^0_0 $ and $ (\varepsilon_C)^0_0:=1_{C^0_0} $. Furthermore, we define $ F^1_0:=C^1_0 $, $\delta_0:=d_0 $, $\gamma_0:=c_0 $, $ \varepsilon_0:=e_0 $ and $ (\varepsilon_C)^1_0:=1_{C^1_0} $. Next we set $ (\varepsilon_C)^0_1:=e^0 $ and $ F^1_1 $, $ (\varepsilon_C)^1_1 $, $ \delta_1 $ and $ \gamma_1 $ to be the object and morphisms, respectively, in the ``joint pullback'' of $ d_1 $ and $ c_1 $ along $ e^0 $ as displayed in the following diagram: 
	\begin{equation*}
		\begin{tikzcd}
			{F^1_1} & {C^1_1} \\
			{C^0_0\times C^0_0} & {C^0_1\times C^0_1.}
			\arrow["{(\varepsilon_C)^1_1}", from=1-1, to=1-2]
			\arrow["{(\delta_1,\gamma_1)}"', from=1-1, to=2-1]
			\arrow["{e^0\times e^0}"', from=2-1, to=2-2]
			\arrow["{(d_1,c_1)}", from=1-2, to=2-2]
		\end{tikzcd}
	\end{equation*}
	Since $ e^0 \times e^0 $ is a split monomorphism in $\mathscr{C}$, $ (\varepsilon_C)^1_1 $ is a regular monomorphism in $\mathscr{C}$. The morphism $ \varepsilon_1 $ is the unique map such that $ (\delta_1,\gamma_1)\varepsilon_1= (1_{C^0_0},1_{C^0_0}) $ and $ (\varepsilon_C)^1_1\varepsilon_1=e_1e^0 $ induced by the universal property of the pullback $ F^1_1 $. We set $ \delta^1:=d^1(\varepsilon_C)^1_1 $ and $ \gamma^1:=c^1(\varepsilon_C)^1_1 $. Finally, $ \varepsilon^1 $ is the unique morphism such that $ (\delta_1,\gamma_1)\varepsilon^1=(d_0,c_0) $ and $ (\varepsilon_C)^1_1\varepsilon^1=e^1 $ induced by the universal property of the pullback $ F^1_1$. 
	
	It is easily seen that the front part of the diagram \eqref{coreflection} is a double reflexive graph $ \mathbb{F} $ in $ \mathscr{C} $, which is underlying a double groupoid structure. Furthermore, $ \varepsilon_C:\mathbb{F}\to\mathbb{C} $ is a double functor, and it is not difficult to prove that it has the universal property of the counit of the adjunction. 
\end{proof}

Our proof of the main result of this section also relies on the following:

\begin{proposition}\label{prop:actionrepresentablesubcategory}\cite[Proposition 2.9]{gran.gray:2021}
	Let $ \mathscr{C} $ be a semi-abelian action representable category. Moreover, let $ \mathscr{X} $ be a semi-abelian category  and $ \mathsf{H}:\mathscr{X}\to\mathscr{C} $ be a fully faithful left adjoint functor that is protoadditive \cite{proto}, i.e., it preserves kernels of split epimorphisms. Then $ \mathscr{X} $ is action representable.
\end{proposition}

\begin{theorem}\label{main-AR}
	Let $ \mathscr{C} $ be a semi-abelian action representable algebraically coherent category with normalizers. Then $ 2\textrm{-Grpd}(\mathscr{C}) $ is a semi-abelian action representable algebraically coherent category with normalizers. 
\end{theorem}
\begin{proof}
	By Theorem~\ref{thm:Grpd(C)actionrepresentable}, we know that $ \mathrm{Grpd}^2(\mathscr{C}) $ is a semi-abelian action representable algebraically coherent category with normalizers. By Theorem~\ref{Birkhoff} and Lemma \ref{prop:2-Grpd(C)coreflective}, $ 2\textrm{-Grpd}(\mathscr{C}) $ is a coreflective Birkhoff subcategory of $ \mathrm{Grpd}^2(\mathscr{C}) $. We have already observed that $ 2\textrm{-Grpd}(\mathscr{C}) $ is a semi-abelian category. It follows from \cite[Proposition 3.7]{cigoli.gray.vanderlinden:2015} that it is also algebraically coherent. By using the fact that $\textrm{Grpd}^2(\mathscr{C}) $ has normalizers one sees that the same is true for $2\textrm{-Grpd}(\mathscr{C}). $
	Proposition~\ref{prop:actionrepresentablesubcategory} shows that $ 2\textrm{-Grpd}(\mathscr{C}) $ is action representable,  since the inclusion functor $ 2\textrm{-Grpd}(\mathscr{C}) \rightarrow  \mathrm{Grpd}^2(\mathscr{C}) $ is clearly protoadditive (since it has a left adjoint).
\end{proof}

\section*{Statements and Declarations}

The authors have no conflicts of interest to declare that are relevant to the content of this article.


\begin{thebibliography}{10}
	\bibitem{Barr}
	\textsc{M. Barr}, Exact categories, \textit{Springer Lecture Notes in Math.} \textbf{236} (1971), 1--120.
	
	\bibitem{BJK} \textsc{F.~Borceux, G.~Janelidze, and G.~M. Kelly}, {Internal object actions}, \textit{Comment. Math. Univ. Carolin.} \textbf{46} 2 (2005), 235--255.
	
	\bibitem{BB}
	\textsc{D. Bourn}, {From groups to categorial algebra}, Compact Textbooks in Mathematics \textbf{566}, {Birkhauser} (2017).
	
	\bibitem{Bourn-Gran}
	\textsc{D. Bourn and M. Gran}, Central extensions in semi-abelian categories, \textit{J. Pure Appl. Algebra} \textbf{175} (2002), 31--44.
	
	\bibitem{BGJ}
	\textsc{D. Bourn, M. Gran and P.-A. Jacqmin}, On the naturalness of Mal'tsev categories, in: \emph{Joachim Lambek: the interplay of mathematics, logic, and linguistics,} Casadio, Claudia (ed.) et al., Springer, Outstanding Contributions to Logic \textbf{20} (2021), 59--104.
	
	\bibitem{CLP}
	\textsc{A. Carboni, J. Lambek, and M.C. Pedicchio}, Diagram chasing in {M}al'cev categories, \textit{J. Pure Appl. Algebra} \textbf{69} (1991), 271--284.
	
	\bibitem{CKP}
	\textsc{A. Carboni, G.M. Kelly and M.C. Pirovano}, Some remarks on Maltsev and Goursat categories, \textit{Appl. Categ. Structures} \textbf{1} (1993), 385--421.
	
	
	\bibitem{CPP}
	\textsc{A. Carboni, C. Pedicchio, and N. Pirovano}, Internal graphs and internal groupoids in Mal'cev categories, \textit{Canadian Mathematical Society Conference Proceedings} \textbf{13} (1992), 97--109. 
	
	\bibitem{cigoli.gray.vanderlinden:2015}
	\textsc{A. S. Cigoli, J. R. A. Gray and T. Van der Linden}, Algebraically coherent categories, \textit{Theory Appl. Categ.} \textbf{30} 54 (2015), 1864--1905. 
	
	\bibitem{Conduche}
	\textsc{D. Conduch\'{e} and G.J. Ellis}, Quelques propri\'{e}t\'{e}s homologiques des modules pr\'{e}crois\'{e}s, \textit{J. Algebra} \textbf{123} (1989), 327--335.
	
	\bibitem{proto}
	\textsc{T. Everaert and M. Gran}, 
	Protoadditive functors, derived torsion theories and homology, \textit{J. Pure Appl. Algebra} \textbf{219} (2015), 3629--3676.
	
	\bibitem{Everaert-K-T}
	\textsc{T. Everaert, R. W. Kieboom, T. Van der Linden}, Model structures for homotopy of internal categories, \textit{Theory  Appl. Categ.} \textbf{15} (2005),
	66--94. 
	
	\bibitem{gran:1999}
	\textsc{M. Gran}, Internal categories in Mal'cev categories, \textit{J. Pure Appl. Algebra} \textbf{143} (1999), 221--229.
	
	\bibitem{gran.gray:2021}
	\textsc{M. Gran and J. R. A. Gray}, Action representability of the category of internal groupoids, \textit{Theory Appl. Categ.} \textbf{37} 1 (2021), 1--13. 
	
	\bibitem{Gran-Rosicky}
	\textsc{M. Gran and J. Rosick\'y}, 
	{Special reflexive graphs in modular varieties}, \textit{Algebra Universalis} \textbf{52} (2004), 89--102.
	
	\bibitem{GRT}
	\textsc{M. Gran, I. Tchoffo Nguefeu and D. Rodelo}, Some remarks on connectors and groupoids in Goursat categories, \textit{ Logical Methods in Computer Science} \textbf{13} (2017), 1--12.
	
	\bibitem{Gray}
	\textsc{J.R.A Gray}, Normalizers, centralizers and action representability in semi-abelian categories, \textit{ Appl. Categ. Struct.} \textbf{22} (2014), 981--1007.
	
	\bibitem{janelidze:1990}
	\textsc{G. Janelidze}, Internal categories in Mal'cev varieties, preprint, York University in Toronto (1990). 
	
	\bibitem{JMT}
	\textsc{G. Janelidze, L. M\'arki and W. Tholen}, Semi-abelian categories, \textit{J. Pure Appl. Algebra} \textbf{168} (2002), 367--386. 
	
	\bibitem{Ja-Pe}
	\textsc{G. Janelidze and M. C. Pedicchio}, Internal categories and groupoids in congruence modular
	varieties, \textit{J. Algebra} \textbf{193} (1997), 552--570.
	
	\bibitem{Johnstone}
	\textsc{P.T. Johnstone}, Affine categories and naturally Mal'cev categories, \textit{J. Pure Appl. Algebra} \textbf{61} (1989), 251--256. 
	
	\bibitem{Loday}
	\textsc{J.-L. Loday}, Spaces with finitely many non-trivial homotopy groups, \textit{J. Pure Appl. Algebra} \textbf{24} (1982), 179--202. 
	
	\bibitem{MacLane}
	\textsc{S. Mac Lane}, Categories for the Working Mathematician, \textit{Graduate Texts in Mathematics} Second Edition, Springer, New York (1998). 
	
	\bibitem{Pedicchio}
	\textsc{M.C. Pedicchio}, A categorical approach to commutator theory, \textit{J. Algebra} \textbf{177} (1995), 647--657. 
	
	\bibitem{Porter}
	\textsc{T. Porter}, 
	Some categorical results in the theory of crossed modules in commutative algebras, \textit{J. Algebra} \textbf{109} (1987), 415--429.
	
	\bibitem{Smith}
	\textsc{J.D.H. Smith}, Mal'cev varieties, \textit{Lecture Notes in Mathematics} \textbf{554}, Springer-Verlag, Berlin-New York (1976). 
\end{thebibliography}
\end{document}